\documentclass{article}

\usepackage[a4paper,top=2.5cm,bottom=2.5cm,left=2cm,right=2cm,%
includehead,includefoot]{geometry}

\usepackage{amsmath,amssymb,amsfonts,amsthm}
\usepackage{mathtools}
\usepackage{amsbsy}
\numberwithin{equation}{section}
\allowdisplaybreaks

\usepackage{graphicx}
\usepackage{float}
\usepackage{xcolor}
\usepackage{booktabs}
\usepackage{colortbl}
\usepackage{caption}
\usepackage{subcaption}
\usepackage{siunitx}
\usepackage{multirow}
\usepackage{epstopdf}

\usepackage{enumitem}
\usepackage{algorithm}
\usepackage{algpseudocode}
\usepackage{array}
\usepackage{longtable}

\usepackage{hyperref}
\hypersetup{colorlinks,citecolor=blue,linkcolor=blue,breaklinks=true}

\usepackage{fancyhdr}
\usepackage{authblk}

\restylefloat{algorithm}
\floatstyle{plain}

\newtheorem{theorem}{Theorem}[section]
\newtheorem{lemma}{Lemma}[section]

\newtheorem{proposition}[theorem]{Proposition}

\newtheorem{remark}{Remark}

\usepackage{indentfirst}

\begin{document}
	
	\title{A Column-Wise Reflection Method with Safeguarded Anderson Acceleration for the Continuous Sylvester Matrix Equation}

	\author[1]{Yinglian Jin}
	\author[1]{Hailong Zhu \thanks{This work was supported by the 2024 Anhui Provincial Scientific Research Plan Project (Key Project, No. 2024AH050013),  Corresponding author: Hailong Zhu (email: hai-long-zhu@163.com).}}
	\author[1]{Wenyue Feng}
	\affil[1]{School of Mathematics and Statistics, Anhui University of Finance and Economics, Bengbu 233030, China}
		
	
	\date{}
	\maketitle
	
	\begin{abstract}
		We propose the safeguarded Anderson-accelerated reflection (SAAR) method, a two-level iterative method for the continuous Sylvester matrix equation $AX+XB=C$. The core of the method is a fixed-point reformulation: at each outer iteration, the associated linear matrix equation is decomposed into independent column-wise linear systems, which are solved approximately by a reflection iteration. For nonsingular \(A\) and \(n\ge 2\), we prove the convergence of the reflection inner iteration by a spectral-radius argument and extend the result to the matrix equation setting. Since only finitely many inner iterations are used in practice, the overall scheme is analyzed as an inexact fixed-point iteration, and a sufficient condition for linear convergence of the unaccelerated inexact iteration is established in the Frobenius norm. On top of this scheme, SAAR incorporates Anderson acceleration into the outer iteration, while a residual-based safeguard rejects Anderson candidates whose Sylvester residual is larger than that of the standard inexact reflection iterate, thereby enhancing numerical stability. Numerical experiments indicate that SAAR substantially reduces the number of outer iterations and often the CPU time, and it improves empirical robustness with respect to the inner tolerance for the tested problems.
	\end{abstract}
	
	
	
	\section{Introduction}
	
	We consider iterative solutions of the continuous Sylvester equations of the form
	\begin{equation}
		AX+XB=C,
		\label{eq:1.1}
	\end{equation}
	where \(A,B,C\in \mathbb{R}^{n\times n}\) are given matrices and \(X\in \mathbb{R}^{n\times n}\) is the unknown matrix. The continuous Sylvester equation \eqref{eq:1.1} has a unique solution if and only if \(A\)
	and \(-B\) have no common eigenvalues, which will be assumed throughout this paper. In addition to the unique solvability of the Sylvester equation, we assume throughout the analysis and algorithm that \(A\) is nonsingular, since the proposed fixed-point formulation and the reflection-based inner iteration require \(A^{-1}\) to exist.
	The continuous Sylvester equation has numerous applications in control and system theory \cite{Datta2003}, stability of linear systems \cite{HalanayRasvan1993}, signal processing \cite{Anderson1986}, image processing \cite{Bouhamidi2007}, and has attracted considerable attention in the study of matrix equations \cite{Levenberg1993, LiMeiKong2016, BaoLinWei2006, BaiGolubNg2008, KhorsandZak2014}.
	
	Standard methods for the numerical solution of the Sylvester equation\eqref{eq:1.1} include the Bartels-Stewart method \cite{BartelsStewart1972} and the Hessenberg-Schur methods \cite{Quintana1995}, which transform \(A\)
	and \(B\) into triangular or Hessenberg form by an orthogonal similarity transformation and then solve the resulting triangular or quasi-triangular matrix equation by back substitution. These methods are classified
	as direct methods and are applicable and effective for general Sylvester equations of reasonably
	small size.
	When the matrices \(A\) and \(B\) are large and sparse, iterative methods such as the alternating direction implicit (ADI) method \cite{Levenberg1993}, the Smith's
	method \cite{LiMeiKong2016}, the Krylov subspace based algorithms \cite{BaoLinWei2006, ElGuennouniJbilouSadok2003, SalkuyehToutounian2006}, the Hermitian and skew-Hermitian splitting (HSS) method \cite{Bai2011}, the inexact
	variant of HSS (IHSS) iteration method \cite{BaiGolubNg2008}, and nested splitting
	conjugate gradient (NSCG) method \cite{KhorsandZak2014} are often the methods of choice for efficiently and
	accurately solving the Sylvester equation \eqref{eq:1.1}.
	
	The reflection algorithm \cite{GuidaSbordone2023} originates from Cimmino's
	iterative method \cite{Benzi2022}. In \cite{FengZhu2025, GuidaSbordone2023}, the authors presented the reflection method and the orthogonal reflection method for linear systems for nonsingular linear systems. Guida and Sbordone \cite{GuidaSbordone2023} mainly prove the convergence of the reflection method in two-dimensional space, and Feng and Zhu \cite{FengZhu2025}, following the theory of \cite{GuidaSbordone2023}, extend the two-dimensional space to \(n\)-dimensional space and propose the orthogonal reflection method. These methods
	start with an arbitrary point in \(\mathbb{R}^{n}\) that gets reflected with respect to the system's hyperplanes. The centroid of the ensuing collection of points becomes the starting point of the next iteration. The reflection method can also be incorporated into a fixed-point framework for the continuous Sylvester equation. Directly applying vectorization to the Sylvester equation \eqref{eq:1.1} leads to a Kronecker-product linear system of dimension \(n^2 \times n^2\), which is computationally prohibitive for large-scale problems. To mitigate this issue, D. Khojasteh Salkuyeh and F. Toutounian \cite{SalkuyehToutounian2006} developed a fixed-point iteration approach that makes effective use of state-of-the-art solvers for block linear systems. Following this philosophy, we decompose the original equation into \(n\) column-wise linear subsystems and employ the reflection method to solve them.
	
	In this paper, we propose the safeguarded Anderson-accelerated reflection (SAAR) method, a two-level iterative method for solving the Sylvester matrix equation. The method is based on a fixed-point formulation, which generates a sequence of matrix equations. At each outer iteration, the corresponding matrix equation is decomposed into \(n\) independent column-wise linear sub-problems, which are solved approximately by the reflection method. Thus, the fixed-point iteration serves as the outer iteration, while the reflection method is employed as the inner iteration. Since the inner problems are solved only approximately, the resulting scheme can be interpreted as an inexact fixed-point iteration. To further improve the convergence behavior of the outer iteration, Anderson acceleration is incorporated, together with a safeguard strategy to enhance numerical stability. The convergence of the unaccelerated inexact fixed-point iteration is analyzed under suitable assumptions, and numerical experiments are presented to demonstrate its effectiveness. Throughout the convergence analysis of the reflection inner iteration, we assume \(n \ge 2\). The scalar case \(n = 1\) is excluded because the reflection iteration may fail to converge, as shown in Remark~\ref{rem:n=1}.
	
	The organization of this paper is as follows. Section~2 contains brief preliminaries. Section~3 presents the reflection method for the continuous Sylvester equations; Section~4 presents the convergence analysis of the unaccelerated inexact iteration; Section~5 describes the safeguarded Anderson acceleration strategy; Section~6 is devoted to numerical experiments; Finally, conclusions are given in Section~7.
	
	\section{Preliminaries and Notation}
	\label{sec:preliminaries}
	
	We consider the continuous Sylvester matrix equation
	\begin{equation}
		AX+XB=C,
		\label{eq:sylvester-prelim}
	\end{equation}
	where \(A,B,C\in\mathbb{R}^{n\times n}\) are given matrices and
	\(X\in\mathbb{R}^{n\times n}\) is the unknown matrix.
	
	\subsection{Fixed-point formulation}
	
	For later use, we summarize the notation associated with the two-level
	iteration:
	
	\begin{itemize}
		\item \(k\): the index of the outer fixed-point iteration, \(k=0,1,2,\ldots\);
		\item \(\nu\): the index of the inner reflection iteration, \(\nu=0,1,2,\ldots\);
		\item \(X^{(k)}\): the accepted matrix iterate after the \(k\)-th outer iteration;
		\item \(\boldsymbol{x}_j^{(k)}\): the \(j\)-th column of \(X^{(k)}\);
		\item \(D^{(k)}:=C-X^{(k)}B\): the right-hand-side matrix of the inner problem at outer iteration \(k\);
		\item \(X^{(k+1,\nu)}\): the matrix approximation obtained after $\nu$-th inner reflection iterations for solving the matrix equation \eqref{eq:outer_matrix_equation} at the $k$-th outer iteration;
		\item \(\boldsymbol{x}_j^{(k+1,\nu)}\): the \(j\)-th column of \(X^{(k+1,\nu)}\);
		\item \(X^{(k+1,0)}\): the initial guess for the inner iteration at the $k$-th outer iteration;
		\item \(X^{(k+1,\ast)}\): the exact solution of the inner matrix equation;
		\item \(X_{\mathrm{inex}}^{(k+1)}\): the approximate solution of the inner matrix equation obtained via finite inner iterations;
		\item \(X_{\mathrm{ex}}^{(k+1)}:=X^{(k+1,\ast)}\): the exact fixed-point update;
		\item \(X_{\mathrm{AA}}^{(k+1)}\): the Anderson-accelerated candidate;
		\item \(G^{(k)}\): the residual of outer fixed point at the $k$-th outer iteration for Anderson acceleration;
		\item \(R_{\mathrm{Syl}}^{(k)}\): the residual of Sylvester equation at the $k$-th iteration;
		\item \(R_{\mathrm{in}}^{(k+1,0)}\): the initial inner residual at the \(k\)-th outer iteration;
		\item \(R_{\mathrm{in}}^{(k+1,\nu)}\): the inner iteration residual at the \(\nu\)-th step at the $k$-th outer iteration;
		\item \(R_{\mathrm{in}}^{(k+1)}\): the final inner residual after the inner iteration stops at the $k$-th outer iteration.
	\end{itemize}
	
	Given an initial matrix \(X^{(0)}\in\mathbb{R}^{n\times n}\), equation
	\eqref{eq:sylvester-prelim} can be reformulated as the following
	fixed-point iteration:
	\begin{equation}
		AX^{(k+1)} = C-X^{(k)}B =:D^{(k)},
		\qquad k=0,1,2,\ldots.
		\label{eq:fixed-point-linear}
	\end{equation}
	Equivalently,
	\begin{equation}
		X_{\mathrm{ex}}^{(k+1)} = \mathcal{T}\bigl(X^{(k)}\bigr),
		\qquad
		\mathcal{T}(X):=A^{-1}(C-XB).
		\label{eq:fixed-point-map}
	\end{equation}
	
	Let \(X^\ast\) denote the exact solution of \eqref{eq:sylvester-prelim}. Then
	\[
	X^\ast=\mathcal{T}(X^\ast).
	\]
	
	The following result gives a sufficient condition for the convergence of the fixed-point iteration.
	
	\begin{theorem}[{\cite{SalkuyehToutounian2006}: sufficient condition for the convergence of the fixed-point iterative method}]
		\label{thm:fixed-point-convergence}
		Let \(A,B,C\in\mathbb{R}^{n\times n}\), and let \(\|\cdot\|\) be an induced norm. If
		\[
		\|A^{-1}\|\,\|B\|<1,
		\]
		then equation \eqref{eq:sylvester-prelim} has a unique solution, and the sequence \(\{X^{(k)}\}_{k\geq 0}\) generated by \eqref{eq:fixed-point-linear} converges to \(X^\ast\).
	\end{theorem}
	Theorem~\ref{thm:fixed-point-convergence} corresponds to the exact-solve limit ($\eta \to 0$) of the inexact result in Theorem~\ref{thm:main-convergence}; see \eqref{eq:qdef}.
	
	Since the reflection method is used as the inner iteration for solving \eqref{eq:fixed-point-linear}, the inner matrix equation is generally solved only approximately. Consequently, the practical outer iteration can be regarded as an inexact fixed-point iteration.
	
	Let \(X_{\mathrm{inex}}^{(k+1)}\) denote the approximate solution obtained for the inner problem
	\[
	AX=D^{(k)}.
	\]
	We define the approximate fixed-point mapping by
	\[
	\widetilde{\mathcal{T}}\bigl(X^{(k)}\bigr) := X_{\mathrm{inex}}^{(k+1)}.
	\]
	If \(\Delta^{(k)}\) denotes the perturbation introduced by the inexact solution of the inner problem, then
	\[
	X_{\mathrm{inex}}^{(k+1)} = \mathcal{T}\bigl(X^{(k)}\bigr)+\Delta^{(k)}.
	\]
	
	\subsection{Column-wise decomposition of the inner problem}
	
	Equation \eqref{eq:fixed-point-linear} is equivalent to \(n\)
	independent linear systems for the columns of \(X^{(k+1)}\):
	\begin{equation}
		A\boldsymbol{x}_j^{(k+1)} = \boldsymbol{d}_j^{(k)},
		\qquad j=1,2,\ldots,n,
		\label{eq:column-wise-system}
	\end{equation}
	where \(\boldsymbol{d}_j^{(k)}\) denotes the \(j\)-th column of \(D^{(k)}\).
	
	Let
	\[
	A=\begin{pmatrix}\boldsymbol{a}_1^{\mathsf T} \\ \boldsymbol{a}_2^{\mathsf T} \\ \vdots \\ \boldsymbol{a}_n^{\mathsf T}\end{pmatrix},
	\]
	where $\boldsymbol{a}_i^{\mathsf T}$ denotes the $i$-th row of $A$. For fixed
	\(k\) and \(j\), the \(j\)-th column system in \eqref{eq:column-wise-system} can be written componentwise as
	\[
	\langle \boldsymbol{a}_i,\boldsymbol{x}_j^{(k+1)}\rangle = d_{ij}^{(k)},
	\qquad i=1,2,\ldots,n,
	\]
	where \(d_{ij}^{(k)}\) is the \(i\)-th component of \(\boldsymbol{d}_j^{(k)}\).
	
	More generally, consider the linear system
	\begin{equation}
		A\boldsymbol{x}=\boldsymbol{d},
		\label{eq:generic-linear-system}
	\end{equation}
	where \(A\in\mathbb{R}^{n\times n}\) is nonsingular. Each equation in \eqref{eq:generic-linear-system} defines an affine hyperplane
	\[
	H_i = \left\{ \boldsymbol{x}\in\mathbb{R}^n: \langle \boldsymbol{a}_i,\boldsymbol{x}\rangle=\boldsymbol{d}_i \right\},
	\qquad i=1,2,\ldots,n.
	\]
	Because \(A\) is nonsingular, these hyperplanes have a unique common intersection, namely the exact solution \(\boldsymbol{x}^\ast\) of \eqref{eq:generic-linear-system}.
	
	\subsection{The Reflection method for linear system}
	
	In this subsection, we prove that the spectral radius of the iteration matrix is strictly less than one, and hence the method converges monotonically to the exact solution for any initial guess. We begin with two lemmas on its basic properties.
	
	\begin{lemma}[{\cite{GuidaSbordone2023}}]
		\label{lem:reflection-point}
		Let
		\[
		H_i = \left\{ \boldsymbol{x}\in\mathbb{R}^n: \langle \boldsymbol{a}_i,\boldsymbol{x}\rangle=\boldsymbol{d}_i \right\},
		\qquad \boldsymbol{a}_i\neq 0,
		\]
		for \(i=1,2,\ldots,n\). Given an initial point \(\boldsymbol{x}^{(0)}\in\mathbb{R}^n\), the orthogonal projection of \(\boldsymbol{x}^{(0)}\) onto \(H_i\) is
		\[
		\boldsymbol{R}_i = \boldsymbol{x}^{(0)} + \frac{\boldsymbol{d}_i-\langle \boldsymbol{a}_i,\boldsymbol{x}^{(0)}\rangle}{\|\boldsymbol{a}_i\|_2^2}\,\boldsymbol{a}_i.
		\]
		The reflection of \(\boldsymbol{x}^{(0)}\) with respect to \(H_i\) is therefore
		\[
		\boldsymbol{Q}_i = \boldsymbol{x}^{(0)} + 2\frac{d_i-\langle a_i,\boldsymbol{x}^{(0)}\rangle}{\|a_i\|_2^2}\,a_i.
		\]
	\end{lemma}
	
	\begin{lemma}[{\cite{FengZhu2025}}]
		\label{lem:reflection-iteration}
		Let \(\boldsymbol{x}^{(\nu)}\) be the current iterate, and let \(\boldsymbol{P}_i^{(\nu)}\) be its reflection with respect to the hyperplane \(H_i\). Then
		\[
		\boldsymbol{P}_i^{(\nu)} = \boldsymbol{x}^{(\nu)} + 2\frac{\boldsymbol{d}_i-\langle \boldsymbol{a}_i,\boldsymbol{x}^{(\nu)}\rangle}{\|\boldsymbol{a}_i\|_2^2}\,\boldsymbol{a}_i,
		\qquad i=1,2,\ldots,n.
		\]
		The next iterate is defined as the centroid of the reflected points:
		\[
		\boldsymbol{x}^{(\nu+1)} = \frac{1}{n}\sum_{i=1}^n \boldsymbol{P}_i^{(\nu)}.
		\]
		Consequently,
		\[
		\boldsymbol{x}^{(\nu+1)} = \boldsymbol{x}^{(\nu)} + \frac{2}{n} \sum_{i=1}^n \frac{\boldsymbol{d}_i-\langle \boldsymbol{a}_i,\boldsymbol{x}^{(\nu)}\rangle}{\|\boldsymbol{a}_i\|_2^2}\,\boldsymbol{a}_i,
		\qquad \nu=0,1,2,\ldots.
		\]
	\end{lemma}
	The proof of Lemma~\ref{lem:reflection-iteration} can be found in Section~2.3 of \cite{FengZhu2025}.
	
	Let
	$E = \operatorname{diag} \left( \|\boldsymbol{a}_1\|_2^{-2}, \|\boldsymbol{a}_2\|_2^{-2}, \ldots, \|\boldsymbol{a}_n\|_2^{-2} \right)$.
	Then the iteration in Lemma~\ref{lem:reflection-iteration} can be written compactly as
	\[
	\boldsymbol{x}^{(\nu+1)} = \boldsymbol{x}^{(\nu)} + \frac{2}{n} A^{\mathsf T}E\bigl(\boldsymbol{d}-A\boldsymbol{x}^{(\nu)}\bigr).
	\]
	Equivalently,
	\begin{equation}
		\boldsymbol{x}^{(\nu+1)} = \left( I_n-\frac{2}{n}A^{\mathsf T}EA \right)\boldsymbol{x}^{(\nu)} + \frac{2}{n}A^{\mathsf T}E\boldsymbol{d}.
		\label{eq:reflection-affine}
	\end{equation}
	
	The convergence of the reflection iteration is stated below.
	
	\begin{theorem}
		\label{thm:reflection-convergence}
		For all \(n\ge 2\), let \(A\in\mathbb{R}^{n\times n}\) be nonsingular, and \(x^\ast\) be the exact solution of
		\[
		A\boldsymbol{x}=\boldsymbol{d}.
		\]
		For any initial vector \(\boldsymbol{x}^{(0)}\in\mathbb{R}^n\), from Lemma~\ref{lem:reflection-iteration}, we have
		\[
		\boldsymbol{x}^{(\nu+1)} = \boldsymbol{x}^{(\nu)} + \frac{2}{n} \sum_{i=1}^n \frac{\boldsymbol{d}_i-\langle \boldsymbol{a}_i,\boldsymbol{x}^{(\nu)}\rangle}{\|\boldsymbol{a}_i\|_2^2}\,\boldsymbol{a}_i,
		\qquad \nu=0,1,2,\ldots,
		\]
		then the iteration can be written in matrix form as
		\begin{equation}
			\boldsymbol{x}^{(\nu+1)} = \left( I_n-\frac{2}{n}A^{\mathsf T}EA \right)\boldsymbol{x}^{(\nu)} + \frac{2}{n}A^{\mathsf T}E\boldsymbol{d}.
			\label{eq:reflection_matrix_form}
		\end{equation}
		satisfies
		\[
		\|\boldsymbol{x}^{(\nu+1)}-\boldsymbol{x}^\ast\|_2 \leq \|\boldsymbol{x}^{(\nu)}-\boldsymbol{x}^\ast\|_2,
		\qquad \nu=0,1,2,\ldots.
		\]
		with equality holding if and only if \(\boldsymbol{x}^{(\nu)}=\boldsymbol{x}^\ast\).
		
		Moreover, the iteration matrix \(M = I_n - \frac{2}{n} A^{\mathsf T} E A\) satisfies \(\rho(M) < 1\) for any nonsingular \(A\). Consequently,
		\[
		\lim_{\nu \to \infty} \boldsymbol{x}^{(\nu)} = \boldsymbol{x}^{*}.
		\]
	\end{theorem}
	\begin{proof}
		Let
		\[
		M=I_n-\frac{2}{n}A^{\mathsf T}EA.
		\]
		Since $A\boldsymbol{x}^\ast=\boldsymbol{d}$, it follows from \eqref{eq:reflection_matrix_form} that
		\begin{equation}
			\boldsymbol{x}^{(\nu+1)}-\boldsymbol{x}^\ast = M\left(\boldsymbol{x}^{(\nu)}-\boldsymbol{x}^\ast\right).
			\label{eq:error_iteration}
		\end{equation}
		
		Define \(H = A^{\mathsf T} E A\). Since \(E\) is positive definite and \(A\) is nonsingular, \(H\) is symmetric positive definite. Thus, all eigenvalues of \(H\) are positive. Moreover,
		\[
		\operatorname{tr}(H) = \operatorname{tr}(A^{\mathsf T} E A) = \operatorname{tr}(E A A^{\mathsf T}) = \sum_{i=1}^{n} \frac{\|\boldsymbol{a}_i\|_2^2}{\|\boldsymbol{a}_i\|_2^2} = n.
		\]
		Let \(\lambda_i(H)\) be the eigenvalues of \(H\), $i =1,2,\cdots,n$. Since \(\lambda_i(H)>0\) and \(\sum_{i=1}^n \lambda_i(H)=n\), we have \(0<\lambda_i(H)<n\) for each \(i\).
		
		Since \(H\) is symmetric, \(M = I_n - \frac{2}{n}H\) is also symmetric with eigenvalues
		\[
		\mu_i(M) = 1 - \frac{2}{n}\lambda_i(H), \qquad i=1,\dots,n.
		\]
		It follows immediately that \(-1 < \mu_i(M) < 1\), hence \(\rho(M)<1\), and consequently \(\|M\|_2 = \rho(M) < 1\).
		
		Therefore,
		\[
		\|\boldsymbol{x}^{(\nu+1)}-\boldsymbol{x}^\ast\|_2 = \|M (\boldsymbol{x}^{(\nu)}-\boldsymbol{x}^\ast)\|_2 \le \|M\|_2 \|\boldsymbol{x}^{(\nu)}-\boldsymbol{x}^\ast\|_2 < \|\boldsymbol{x}^{(\nu)}-\boldsymbol{x}^\ast\|_2.
		\]
		Finally, recursively applying \eqref{eq:error_iteration} yields $\boldsymbol{x}^{(\nu+1)}-\boldsymbol{x}^\ast = M^{\nu+1}\left(\boldsymbol{x}^{(0)}-\boldsymbol{x}^\ast\right)$. Since \(\rho(M) < 1\), we have \(M^{\nu} \to 0\), and hence \(\lim_{\nu\to\infty} \boldsymbol{x}^{(\nu)} = \boldsymbol{x}^{*}$. This completes the proof.
	\end{proof}
	
	\begin{remark}
		\label{rem:n=1}
		When $n=1$, the only eigenvalue of matrix $H$ is $1$, and hence the eigenvalue of matrix $M$ is $-1$, which generally does not converge. The following is a counterexample.
		
		We show that Theorem~\ref{thm:reflection-convergence} fails for \(n=1\). Take
		\[
		A=[1],\qquad d=[1],
		\]
		so \(x^*=1\). The reflection iteration becomes
		\[
		x^{(\nu+1)} = x^{(\nu)}+2(1-x^{(\nu)}) = 2-x^{(\nu)}.
		\]
		With \(x^{(0)}=0\), this gives
		\[
		0,\ 2,\ 0,\ 2,\ldots,
		\]
		which does not converge to \(x^*=1\). Equivalently, we have
		\[
		x^{(\nu+1)}-1=-(x^{(\nu)}-1),
		\qquad
		\|x^{(\nu+1)}-1\|_2=\|x^{(\nu)}-1\|_2.
		\]
		Moreover, since \(H=A^{\mathrm T}EA=[1]\), its only eigenvalue is \(1\), and hence the only eigenvalue of \(M=I-2H\) is \(-1\). Thus
		\[
		M=[-1],\qquad \rho(M)=1,
		\]
		contradicting \(\rho(M)<1\) in Theorem~\ref{thm:reflection-convergence}. Hence Theorem~\ref{thm:reflection-convergence} requires \(n\ge2\).
	\end{remark}
	
	\section{The Reflection Method for Sylvester Matrix Equation}
	\label{sec:reflection_sylvester}
	
	In this section, we will apply the reflection method to the Sylvester matrix equation \eqref{eq:sylvester-prelim}. Given an initial matrix \(X^{(0)}\in\mathbb{R}^{n\times n}\), we use the fixed-point form
	\begin{equation}
		AX^{(k+1)}=D^{(k)},
		\qquad k=0,1,2,\ldots.
		\label{eq:outer_matrix_equation}
	\end{equation}
	where $D^{(k)}=C-X^{(k)}B.$ Then \eqref{eq:outer_matrix_equation} is equivalent to the following \(n\) linear systems:
	\begin{equation}
		A\boldsymbol{x}_j^{(k+1)}=\boldsymbol{d}_j^{(k)},
		\qquad j=1,2,\ldots,n.
		\label{eq:column_systems}
	\end{equation}
	
	For each column system in \eqref{eq:column_systems}, we apply the reflection iteration. Let \(\boldsymbol{x}_j^{(k+1,\nu)}\) denote the \(\nu\)-th inner approximation to the \(j\)-th column of \(X^{(k+1)}\) at the \(k\)-th outer iteration. Then
	\begin{equation}
		\label{eq:column_reflection_iteration}
		\begin{split}
			\boldsymbol{x}_j^{(k+1,\nu+1)}
			&= \boldsymbol{x}_j^{(k+1,\nu)} + \frac{2}{n}A^{\mathrm{T}}E \left( \boldsymbol{d}_j^{(k)}-A\boldsymbol{x}_j^{(k+1,\nu)} \right)
			\\
			&= \left( I-\frac{2}{n}A^{\mathrm{T}}EA \right)\boldsymbol{x}_j^{(k+1,\nu)} + \frac{2}{n}A^{\mathrm{T}}E\boldsymbol{d}_j^{(k)},
			\qquad \nu=0,1,2,\ldots.
		\end{split}
	\end{equation}
	where $E$ is the diagonal matrix
	$E= \operatorname{diag}\left( {\|\boldsymbol{a}_1\|_2^{-2}},{\|\boldsymbol{a}_2\|_2^{-2}},\ldots,{\|\boldsymbol{a}_n\|_2^{-2}} \right)$
	and $ \boldsymbol{x}_j^{(k+1,\nu+1)}$ denotes the $(\nu+1)$-th inner iterate for the $j$-th column of $X^{(k+1)}.$
	
	\begin{proposition}
		\label{prop:column-inner-convergence}
		For all $n\ge 2$, let \(A\in\mathbb{R}^{n\times n}\) be nonsingular, and \(x_j^{(k+1,*)}\) be the exact solution of
		\[
		A\boldsymbol{x}_j^{(k+1)}=\boldsymbol{d}_j^{(k)},\qquad j=1,2,\ldots,n.
		\]
		For any initial vector \(x_j^{(k+1,0)}\in\mathbb{R}^n\), consider the reflection iteration \eqref{eq:column_reflection_iteration}, which satisfies
		\[
		\|\boldsymbol{x}_j^{(k+1,\nu+1)}-\boldsymbol{x}_j^{(k+1,*)}\|_2 \leq \|\boldsymbol{x}_j^{(k+1,\nu)}-\boldsymbol{x}_j^{(k+1,*)}\|_2,
		\qquad \nu=0,1,2,\ldots.
		\]
		with equality holding if and only if \(\boldsymbol{x}_j^{(k+1,\nu)}=\boldsymbol{x}_j^{(k+1,*)}\).
		Moreover, the iteration matrix \(M = I_n - \frac{2}{n} A^{\mathsf T}  E A\) satisfies \(\rho(M) < 1\) for any nonsingular \(A\). Consequently,
		\[
		\lim_{\nu \to \infty} \boldsymbol{x}_j^{(k+1,\nu)} = \boldsymbol{x}_j^{(k+1,*)}.
		\]
	\end{proposition}
	
	\begin{proof}
		The result follows directly from Theorem~\ref{thm:reflection-convergence} by taking
		\[
		\boldsymbol{d}=\boldsymbol{d}_j^{(k)},\qquad \boldsymbol{x}^{(\nu)}=\boldsymbol{x}_j^{(k+1,\nu)}.
		\]
	\end{proof}
	
	By Proposition~\ref{prop:column-inner-convergence}, the sequence $\boldsymbol{x}_j^{(k+1,\nu+1)}$ also converges to the exact solution \(\boldsymbol{x}_j^{(k+1,*)}\), that is,
	\[
	\lim_{\nu\to\infty}\boldsymbol{x}_j^{(k+1,\nu)}=\boldsymbol{x}_j^{(k+1,*)},\quad j=1,2,\cdots,n.
	\]
	
	Stacking all columns together gives the matrix form of the reflection iteration for the Sylvester equation
	\begin{equation}
		X^{(k+1,\nu+1)} = X^{(k+1,\nu)} + \frac{2}{n}A^{\mathsf T}E \left( D^{(k)}-AX^{(k+1,\nu)} \right)
		\label{eq:matrix_reflection_iteration}
	\end{equation}
	where $X^{(k+1,\nu+1)}$ denotes the matrix approximation obtained after $(\nu+1)$-th inner reflection iteration for solving the matrix equation \eqref{eq:outer_matrix_equation} at the $k$-th outer iteration.
	
	Equivalently,
	\[
	X^{(k+1,\nu+1)} = \left(I_n-\frac{2}{n}A^{\mathsf T}EA\right)X^{(k+1,\nu)} + \frac{2}{n}A^{\mathsf T}ED^{(k)}.
	\]
	
	\begin{theorem}
		\label{thm:matrix_inner_convergence}
		For all $n\ge 2$, let \(A\in\mathbb{R}^{n\times n}\) be nonsingular, and \(X^{(k+1,*)}\) be the exact solution of
		\[
		AX ^{(k+1)}=D^{(k)}.
		\]
		For any initial matrix \(X^{(k+1,0)} \in\mathbb{R}^{n\times n}\), consider the reflection iteration \eqref{eq:matrix_reflection_iteration}, which satisfies
		\[
		\|X^{(k+1,\nu+1)}-X^{(k+1,*)}\|_F \leq \|X^{(k+1,\nu)}-X^{(k+1,*)}\|_F,
		\qquad \nu=0,1,2,\ldots.
		\]
		with equality holding if and only if \(X^{(k+1,\nu)}=X^{(k+1,*)}\).
		Moreover, the iteration matrix \(M = I_n - \frac{2}{n} A^{\mathsf T} E A\) satisfies \(\rho(M) < 1\) for any nonsingular \(A\). Consequently,
		\[
		\lim_{\nu \to \infty} X^{(k+1,\nu)} = X^{(k+1,*)}.
		\]
	\end{theorem}
	
	\begin{proof}
		\[
		\begin{aligned}
			\|X^{(k+1,\nu+1)} - X^{(k+1,*)}\|_F &= \left\|M\left(X^{(k+1,\nu)} - X^{(k+1,*)}\right)\right\|_F \\
			&\le \|M\|_2 \left\|X^{(k+1,\nu)} - X^{(k+1,*)}\right\|_F.
		\end{aligned}
		\]
		
		By Theorem~\ref{thm:reflection-convergence}, $\|M\|_2 = \rho(M) < 1$. Hence,
		\[
		\|X^{(k+1,\nu)} - X^{(k+1,*)}\|_F \le \|M\|_2^\nu \left\|X^{(k+1,0)} - X^{(k+1,*)}\right\|_F \to 0,
		\]
		and the inequality is strict whenever $X^{(k+1,\nu)} \neq X^{(k+1,*)}$. Therefore equality holds if and only if $X^{(k+1,\nu)} = X^{(k+1,*)}$.
		Consequently
		\[
		\lim_{\nu \to \infty} X^{(k+1,\nu)} = X^{(k+1,*)}.
		\]
	\end{proof}
	Therefore, we have shown that for each fixed outer iteration $k$, the matrix equation \eqref{eq:outer_matrix_equation} converges using the reflection method, and it converges to the exact solution of the equation \(AX=D^{(k)}\).
	
	\section{Convergence analysis}
	\label{sec:convergence}
	
	In practical computations, the matrix equation
	\begin{equation}
		A X = D^{(k)}, \qquad D^{(k)} = C - X^{(k)} B,
		\label{eq:4.1}
	\end{equation}
	is not solved exactly at each outer iteration. Therefore, the proposed reflection-based method should be regarded as an inexact fixed-point iteration. Since $A$ is nonsingular, the continuous Sylvester equation \eqref{eq:sylvester-prelim} can be written in the fixed-point form
	\[
	X = \mathcal{T}(X) := A^{-1}(C - X B).
	\]
	
	Given $ X^{(k+1,0)}=X^{(k)}$, the exact update is $X_{\mathrm{ex}}^{(k+1)} = \mathcal{T}(X^{(k)}) = A^{-1} D^{(k)}$. In our method, we compute only an approximation $X_{\mathrm{inex}}^{(k+1)}$ via a finite number of inner iterations. Define the final inner residual
	\[
	R_{\mathrm{in}}^{(k+1)} :=AX_{\mathrm{inex}}^{(k+1)} - D^{(k)}.
	\]
	and the corresponding initial inner residual
	\[
	R_{\mathrm{in}}^{(k+1,0)} := A X^{(k+1,0)} - D^{(k)}.
	\]
	Then
	\[
	X_{\mathrm{inex}}^{(k+1)} = A^{-1} D^{(k)} + A^{-1} R_{\mathrm{in}}^{(k+1)} = \mathcal{T}(X^{(k)}) + \Delta^{(k)},
	\]
	where $\Delta^{(k)} := A^{-1} R_{\mathrm{in}}^{(k+1)}$.
	
	Taking the next outer iterate as $X^{(k+1)} = X_{\mathrm{inex}}^{(k+1)}$, we obtain the perturbed iteration
	\begin{equation}
		X^{(k+1)} = \mathcal{T}(X^{(k)}) + \Delta^{(k)}.
		\label{eq:4.5}
	\end{equation}
	Since $X^* = \mathcal{T}(X^*)$, subtracting $X^*$ from \eqref{eq:4.5} yields
	\begin{equation}
		X^{(k+1)} - X^* = -A^{-1} \bigl(X^{(k)} - X^*\bigr) B + A^{-1} R_{\mathrm{in}}^{(k+1)}.
		\label{eq:4.6}
	\end{equation}
	
	The following elementary bound will be used in the convergence proof.
	
	\begin{lemma}
		\label{lem:4.1}
		Let \(A,B,C\in\mathbb{R}^{n\times n}\), where $A$ is nonsingular, and let \(X^*\) be the solution of \eqref{eq:sylvester-prelim}.
		For the \(k\)-th outer iterate \(X^{(k)}\), define $X^{(k+1,0)}:=X^{(k)}$, and set
		\[
		R_{\mathrm{in}}^{(k+1,0)}:=AX^{(k+1,0)}-D^{(k)},
		\]
		then
		\[
		\left\|R_{\mathrm{in}}^{(k+1,0)}\right\|_F \leq \left(\|A\|_2+\|B\|_2\right) \left\|X^{(k)}-X^*\right\|_F.
		\]
	\end{lemma}
	\begin{proof}
		Since \(AX^*+X^*B=C\), we have
		\[
		R_{\mathrm{in}}^{(k+1,0)} =A(X^{(k)}-X^*)+(X^{(k)}-X^*)B.
		\]
		Taking Frobenius norms and using spectral-norm compatibility yields
		\[
		\begin{aligned}
			\|R_{\mathrm{in}}^{(k+1,0)}\|_F
			&\le \|A(X^{(k)}-X^*)\|_F + \|(X^{(k)}-X^*)B\|_F \\
			&\le (\|A\|_2+\|B\|_2)\|X^{(k)}-X^*\|_F.
		\end{aligned}
		\]
	\end{proof}
	
	We now state the main convergence result.
	
	\begin{theorem}
		\label{thm:main-convergence}
		Let $A,B,C \in \mathbb{R}^{n\times n}$, where $A$ is nonsingular, and let $X^*$ be the unique solution of \eqref{eq:sylvester-prelim}. Consider the inexact fixed-point iteration
		\[
		X^{(k+1)} = X_{\mathrm{inex}}^{(k+1)}, \qquad
		X_{\mathrm{inex}}^{(k+1)} \approx A^{-1}(C - X^{(k)}B),
		\]
		where the inner approximation satisfies, for some $0<\eta<1$,
		\begin{equation}
			\|R_{\mathrm{in}}^{(k+1)} \|_F \le \eta \| R_{\mathrm{in}}^{(k+1,0)} \|_F,
			\qquad D^{(k)} = C - X^{(k)}B.
			\label{eq:stopping}
		\end{equation}
		Define
		\begin{equation}
			q := \| A^{-1} \|_2 \| B \|_2 + \eta \| A^{-1} \|_2 \bigl( \| A \|_2 + \| B \|_2 \bigr).
			\label{eq:qdef}
		\end{equation}
		If $q<1$, then the iteration converges linearly to $X^*$ in the Frobenius norm:
		\[
		\| X^{(k)} - X^* \|_F \le q^{k} \, \| X^{(0)} - X^* \|_F,
		\qquad k = 0,1,2,\dots
		\]
		In particular, $\displaystyle \lim_{k\to\infty} X^{(k)} = X^*$.
	\end{theorem}
	
	\begin{proof}
		From \eqref{eq:4.6}, we have
		\[
		X^{(k+1)} - X^* = -A^{-1} (X^{(k)} - X^*) B + A^{-1} R_{\mathrm{in}}^{(k+1)}.
		\]
		Taking the Frobenius norm and applying the compatibility of the spectral norm,
		\[
		\| X^{(k+1)} - X^* \|_F
		\le \|A^{-1}\|_2 \|B\|_2 \|X^{(k)} - X^*\|_F
		+ \|A^{-1}\|_2 \|R_{\mathrm{in}}^{(k+1)}\|_F.
		\]
		Using the stopping criterion \eqref{eq:stopping} and Lemma~\ref{lem:4.1},
		\[
		\|A X_{\mathrm{inex}}^{(k+1)} - D^{(k)}\|_F
		\le \eta \|A X^{(k)} - D^{(k)}\|_F
		\le \eta \bigl( \|A\|_2 + \|B\|_2 \bigr) \|X^{(k)} - X^*\|_F.
		\]
		Combining both inequalities yields
		\[
		\| X^{(k+1)} - X^* \|_F \le q \| X^{(k)} - X^* \|_F,
		\]
		with $q$ defined in \eqref{eq:qdef}. Since $q<1$, induction gives
		\[
		\| X^{(k)} - X^* \|_F \le q^{k} \| X^{(0)} - X^* \|_F,
		\]
		which proves linear convergence and hence $X^{(k)} - X^* \to 0$.
	\end{proof}
	
	\begin{remark}
		The condition $q<1$ in Theorem~\ref{thm:main-convergence} is sufficient but not necessary; Examples~1 and~3 fail it with $\eta=0.1$ for different reasons. Let
		$M=I-\frac{2}{n}A^{\mathrm T}EA$ and let $\nu_k\ge1$ be the number of inner steps at outer iteration $k$. Since
		\[
		X^{(k+1,\nu_k)}-X^{(k)}=(I_n-M^{\nu_k})(\mathcal{T}(X^{(k)})-X^{(k)})
		\]
		and
		\[
		\mathcal{T}(X^{(k)})-X^\ast=-A^{-1}(X^{(k)}-X^\ast)B,
		\]
		the unaccelerated update satisfies
		\[
		X^{(k+1)}-X^\ast = M^{\nu_k}(X^{(k)}-X^\ast) - (I-M^{\nu_k})A^{-1}(X^{(k)}-X^\ast)B,
		\]
		which preserves the inner-error structure, unlike the estimate in Theorem~\ref{thm:main-convergence}.
		
		In Example~1, $q=(1+2\eta)\kappa_2(A)>1$ for every $\eta>0$. With $X^\ast=I_n$ and $X^{(0)}=0$, the exact outer iteration alternates between $0$ and $2I_n$, i.e., $X^{(k)}-X^\ast=(-1)^{k+1}I_n$. Finite inner iterations may damp this oscillation: if ideally $M^{\nu_k}=\theta I_n$ with $0<\theta<1$, scalar errors obey
		\[
		X^{(k+1)}-X^\ast=(2\theta-1)(X^{(k)}-X^\ast)
		\]
		with $|2\theta-1|<1$.
		
		The stopping rule, however, implies neither a scalar outer factor nor $M^{\nu_k}\approx\varepsilon_{\mathrm{in}}I_n$, and as the inner accuracy increases, $M^{\nu_k}\to0$ and the update approaches the nonconvergent exact iteration. Moreover, nonscalar errors may excite unstable directions: for real eigenvalues $\lambda_i,\lambda_j$ of $A$ with right and left eigenvectors $u_i,y_j$, i.e., $Au_i=\lambda_i u_i$ and $y_j^{\mathrm{T}}A=\lambda_j y_j^{\mathrm{T}}$, the exact update acts on $u_i y_j^{\mathrm{T}}$ as multiplication by $-\lambda_j/\lambda_i$, amplified whenever $|\lambda_j|>|\lambda_i|$. These observations explain the tolerance sensitivity of Example~1 qualitatively only.
		
		In Example~3, by contrast, the exact outer map is contractive. Let $s = 2^{-10}$. Since
		\[
		\|B\|_2 \le s + (r + s)(n - 1) < 0.108
		\]
		and $\|A^{-1}\|_2 < 1.008$, we have
		\[
		\|A^{-1}\|_2\|B\|_2 < 0.109 < 1.
		\]
		Hence Theorem~\ref{thm:fixed-point-convergence} applies to the exact fixed-point iteration. The sufficient condition in Theorem~\ref{thm:main-convergence} fails for $\eta = 0.1$ mainly because of the residual bound involving $\|A\|_2$. Indeed, $Ae_1 = e_1$ implies $\|A^{-1}\|_2 \ge 1$, while
		\[
		Ae_n = ne_n + r\sum_{i=1}^{n-1} e_i
		\]
		implies $\|A\|_2 \ge n = 100,$ where $e_i$ denotes the $i$-th standard basis vector of $\mathbb{R}^n$. Therefore
		\[
		\eta\|A^{-1}\|_2\|A\|_2 \ge 10
		\]
		when $\eta = 0.1$. On the other hand, using the estimates $\|A\|_2 \le n + r(n - 1) < 100.01$, $\|A^{-1}\|_2 < 1.008$, and $\|B\|_2 < 0.108$, we obtain
		\[
		q < 0.109 + 100.92\eta,
		\]
		so Theorem~\ref{thm:main-convergence} applies only for approximately $\eta \lesssim 0.0088$.
		
		In summary, Example~1 illustrates that, for certain initial errors and inner stopping rules, finite inner iterations may have a relaxation effect and may stabilize an otherwise nonconvergent exact outer update, while Example~3 shows the conservativeness of the inner-error bound in Theorem~\ref{thm:main-convergence}. Both concern the unaccelerated method; the Anderson-accelerated results remain numerical evidence rather than consequences of Theorem~\ref{thm:main-convergence}.
	\end{remark}
	
	The complete unaccelerated reflection method is summarized in Algorithm~\ref{alg:reflection}.
	
	\begin{algorithm}[H]
		\caption{The Reflection Method for $AX+XB=C$}
		\label{alg:reflection}
		\begin{algorithmic}[1]
			\State \textbf{Input:} $A,B,C\in\mathbb{R}^{n\times n}$, initial matrix $X^{(0)}$, maximum outer iteration number $k_{\max}$, maximum inner iteration number $\nu_{\max}$, outer tolerance $\varepsilon$, and inner tolerance $\varepsilon_{\mathrm{in}}$
			\State Let $\boldsymbol{a}_i^{\mathsf T}$ denote the $i$-th row of $A$
			\State Compute $E=\operatorname{diag}\left(\|\boldsymbol{a}_1\|_2^{-2},\|\boldsymbol{a}_2\|_2^{-2},\ldots,\|\boldsymbol{a}_n\|_2^{-2}\right).$
			\State Compute $R_{\text{Syl}}^{(0)} = C - A X^{(0)} - X^{(0)} B$.
			\If{$\|R_{\text{Syl}}^{(0)}\|_F = 0$}
			\State \Return $X^{(0)}$
			\EndIf
			\For{$k=0,1,\ldots,k_{\max}-1$}
			\State Compute the right-hand side of the inner problem $D^{(k)}=C-X^{(k)}B$
			\State Set the initial inner iterate $X^{(k+1,0)}=X^{(k)}$
			\State Compute the initial inner residual $R_{\mathrm{in}}^{(k+1,0)}=AX^{(k+1,0)}-D^{(k)}$
			\State Set $\nu_{actual} = \nu_{\max}.$
			\For{$\nu=0,1,\ldots,\nu_{\max}-1$}
			\State Perform one matrix reflection iteration
			\[
			X^{(k+1,\nu+1)} = X^{(k+1,\nu)} +\frac{2}{n}A^{\mathsf T}E \left(D^{(k)}-AX^{(k+1,\nu)}\right)
			\]
			\State Compute the new inner residual $R_{\mathrm{in}}^{(k+1,\nu+1)} = AX^{(k+1,\nu+1)}-D^{(k)}$
			\If{$\|R_{\rm in}^{(k+1,\nu+1)}\|_F \leq \varepsilon_{\rm in} \|R_{\rm in}^{(k+1,0)}\|_F$}
			\State Accept the inner iterate $\nu_{actual} = \nu + 1$.
			\State \textbf{break}
			\EndIf
			\EndFor
			\State Set the inner solution $X_{\mathrm{inex}}^{(k+1)} = X^{(k+1,\nu_{actual})}$
			\State Perform the outer update $X^{(k+1)} = X_{\mathrm{inex}}^{(k+1)}$
			\State Compute the Sylvester residual $R_{\mathrm{Syl}}^{(k+1)} = C-AX^{(k+1)}-X^{(k+1)}B$
			\If{$\left\|R_{\mathrm{Syl}}^{(k+1)}\right\|_{F} \leq \varepsilon\left\|R_{\mathrm{Syl}}^{(0)}\right\|_{F}$}
			\State \Return $X^{(k+1)}$
			\EndIf
			\EndFor
			\State \Return $X^{(k_{\max})}$
		\end{algorithmic}
	\end{algorithm}
	
	\section{Anderson-Accelerated Reflection Method}
	\label{sec:anderson}
	In this section, we apply Anderson acceleration \cite{WalkerNi2011} to the inexact fixed-point iteration in order to improve the practical convergence speed of the outer iteration. According to Section~\ref{sec:convergence}, the basic reflection scheme is an inexact fixed-point iteration.
	
	With the inner initialization and stopping rule of Algorithm~\ref{alg:reflection}, define the approximate fixed-point mapping
	\[
	\widetilde{\mathcal{T}}(X^{(k)}) := X_{\mathrm{inex}}^{(k+1)},
	\]
	where $X^{(k+1)}_{\mathrm{inex}}$ is obtained after a finite number of inner reflection iterations for solving the matrix equation \eqref{eq:4.1}, that is, an approximate solution of \eqref{eq:4.1}.
	
	Thus, $\widetilde{\mathcal{T}}$ is a practical approximation of the exact map $\mathcal{T}(X) = A^{-1}(C - XB)$.
	
	To accelerate the outer iteration, we apply Anderson acceleration to
	\[
	X^{(k+1)} = \widetilde{\mathcal{T}}(X^{(k)}).
	\]
	
	Given the initial matrix $X^{(0)}$, define the relative Sylvester residual of any iterate $X$ by
	\[
	\operatorname{Res}(X) = \frac{\|C-AX-XB\|_F}{\|R_{\mathrm{Syl}}^{(0)}\|_F},
	\]
	where $R_{\mathrm{Syl}}^{(0)}=C-AX^{(0)}-X^{(0)}B.$
	
	Define the fixed-point residual by
	\[
	G^{(k)} = \widetilde{\mathcal{T}}(X^{(k)}) - X^{(k)} = X_{\mathrm{inex}}^{(k+1)} - X^{(k)}.
	\]
	
	Let \(m \ge 0\) be the Anderson window size and set
	\[
	m_k = \min\{m,k\}.
	\]
	At the \(k\)-th outer iteration, the coefficients \(\alpha_0^{(k)}, \alpha_1^{(k)}, \ldots, \alpha_{m_k}^{(k)}\) are determined by solving the constrained least-squares problem
	\[
	\min_{\alpha_0,\alpha_1,\ldots,\alpha_{m_k}} \left\| \sum_{i=0}^{m_k} \alpha_i G^{(k-m_k+i)} \right\|_F,
	\qquad \text{s.t.} \quad \sum_{i=0}^{m_k} \alpha_i = 1.
	\]
	
	The Anderson candidate is then defined by
	\[
	X_{\mathrm{AA}}^{(k+1)} = \sum_{i=0}^{m_k} \alpha_i^{(k)} \widetilde{\mathcal{T}}(X^{(k-m_k+i)}).
	\]
	
	In the implementation, the pairs $\{G^{(i)}, \widetilde{T}(X^{(i)})\}$ are stored for the most recent $m + 1$ accepted outer iterates.
	
	For numerical stability, we use a safeguarded strategy. The safeguard rejects Anderson candidates whose Sylvester residual is larger than that of the standard inexact reflection candidate $X^{(k+1)}_{\mathrm{inex}}$, i.e.,
	\[
	\operatorname{Res}\bigl(X_{\mathrm{AA}}^{(k+1)}\bigr) \le \operatorname{Res}\bigl(X_{\mathrm{inex}}^{(k+1)}\bigr).
	\]
	If this condition holds, we set
	\[
	X^{(k+1)} = X_{\mathrm{AA}}^{(k+1)};
	\]
	otherwise, we reject the Anderson candidate and set
	\[
	X^{(k+1)} = X_{\mathrm{inex}}^{(k+1)}.
	\]
	
	When $m = 0$, the constrained least-squares problem gives $\alpha_0 = 1$, so $X_{\text{AA}}^{(k+1)} = X_{\text{inex}}^{(k+1)}$. Hence Algorithm~\ref{alg:AA-reflection} reduces to the unaccelerated reflection method.
		\begin{algorithm}[!ht]
		\small
		\caption{Safeguarded Anderson-accelerated reflection method for $AX+XB=C$}
		\label{alg:AA-reflection}
		\begin{algorithmic}[1]
			\State \textbf{Input:} initial matrix $X^{(0)}$, Anderson window size $m$, maximum outer iteration number $k_{\max}$, maximum inner iteration number $\nu_{\max}$, tolerances $\varepsilon$ and $\varepsilon_{\rm in}$.
			\State Let $\boldsymbol{a}_i^{\mathsf T}$ denote the $i$-th row of $A$
			\State $E=\operatorname{diag}\left(\|\boldsymbol{a}_1\|_2^{-2},\|\boldsymbol{a}_2\|_2^{-2},\ldots,\|\boldsymbol{a}_n\|_2^{-2}\right).$
			\State Compute $R_{\text{Syl}}^{(0)} = C - A X^{(0)} - X^{(0)} B$.
			\If{$\|R_{\text{Syl}}^{(0)}\|_F = 0$}
			\State \Return $X^{(0)}$
			\EndIf
			\For{$k=0,1,\ldots,k_{\max}-1$}
			\State $D^{(k)}=C-X^{(k)}B.$
			\State $X^{(k+1,0)}=X^{(k)}.$
			\State $R_{\rm in}^{(k+1,0)}=AX^{(k+1,0)}-D^{(k)}.$
			\State Set $\nu_{actual} = \nu_{\max}.$
			\For{$\nu=0,1,\ldots,\nu_{\max}-1$}
			\State Perform one inner reflection iteration
			\[
			X^{(k+1,\nu+1)} = X^{(k+1,\nu)} + \frac{2}{n}A^{\mathsf T}E \left( D^{(k)}-AX^{(k+1,\nu)} \right).
			\]
			\State $R_{\rm in}^{(k+1,\nu+1)} = AX^{(k+1,\nu+1)}-D^{(k)}.$
			\If{$\|R_{\rm in}^{(k+1,\nu+1)}\|_F \leq \varepsilon_{\rm in} \|R_{\rm in}^{(k+1,0)}\|_F$}
			\State Accept the inner iterate $\nu_{actual} = \nu + 1$.
			\State \textbf{break}
			\EndIf
			\EndFor
			\State $X_{\rm inex}^{(k+1)} = X^{(k+1,\nu_{actual})} = \widetilde{\mathcal{T}}\left(X^{(k)}\right).$
			\State $G^{(k)} = \widetilde{\mathcal{T}}\left(X^{(k)}\right)-X^{(k)} = X_{\rm inex}^{(k+1)}-X^{(k)}.$
			\State Set $m_k=\min\{m,k\}.$
			\State Compute the coefficients $\alpha_0^{(k)},\alpha_1^{(k)},\ldots,\alpha_{m_k}^{(k)}$ by solving the constrained least-squares problem
			\[
			\min_{\alpha_0,\alpha_1,\ldots,\alpha_{m_k}} \left\| \sum_{i=0}^{m_k} \alpha_i G^{(k-m_k+i)} \right\|_F, \qquad \text{s.t.} \quad \sum_{i=0}^{m_k}\alpha_i=1.
			\]
			\State Compute the Anderson candidate
			\[
			X_{\rm AA}^{(k+1)} = \sum_{i=0}^{m_k} \alpha_i^{(k)} \widetilde{\mathcal{T}}\left(X^{(k-m_k+i)}\right).
			\]
			\If{$\operatorname{Res}\left(X_{\rm AA}^{(k+1)}\right) \leq \operatorname{Res}\left(X_{\rm inex}^{(k+1)}\right)$}
			\State $X^{(k+1)}=X_{\rm AA}^{(k+1)}.$
			\Else
			\State $X^{(k+1)}=X_{\rm inex}^{(k+1)}.$
			\EndIf
			\State $R_{\mathrm{Syl}}^{(k+1)} = C-AX^{(k+1)}-X^{(k+1)}B.$
			\If{$\|R_{\mathrm{Syl}}^{(k+1)}\|_F \leq \varepsilon \|R_{\mathrm{Syl}}^{(0)}\|_F$}
			\State \Return $X^{(k+1)}$
			\EndIf
			\EndFor
			\State \Return $X^{(k_{\max})}$
		\end{algorithmic}
	\end{algorithm}
	
	\section{Numerical Experiments}
	
	In this section, several numerical experiments are presented to illustrate the effectiveness of the proposed reflection method and the Anderson-accelerated reflection method for solving the Sylvester matrix equation
	\[
	AX+XB=C.
	\]
	
	All experiments are performed in double precision using MATLAB R2024b. The initial guess was chosen as the zero matrix $X^{(0)}=0$. In all experiments, the exact solution was set to $X^*=I_n$, and the right-hand side was defined by $C=AX^*+X^*B$. The outer iteration is terminated when
	\[
	\frac{\|R_{\mathrm{Syl}}^{(k+1)}\|_F}{\|R_{\mathrm{Syl}}^{(0)}\|_F} \le 10^{-6},
	\]
	where $R_{\mathrm{Syl}}^{(k)} = C - A X^{(k)} - X^{(k)} B$ is the residual of the $k$-th iterate. Unless otherwise stated, the inner stopping criterion is set to
	\[
	\frac{\|R_\text{in}^{(k+1,\nu+1)}\|_F }{\|R_\text{in}^{(k+1,0)}\|_F} \le 0.1,
	\]
	where $R_{\mathrm{in}}^{(k+1,\nu+1)} = A X^{(k+1,\nu+1)} - D^{(k)}.$ In Table~\ref{tab:iteration_data}, $\varepsilon_\text{in}$ is varied while the outer tolerance remains fixed at $10^{-6}$.
	
	In the following tables, ``Out-Itr'' denotes the number of iterations or the number of total outer iteration steps, ``Inner-Itr'' denotes the total number of inner reflection iterations, ``Res-Norm'' denotes the relative residual norm, and ``CPU Time'' denotes the computational time in seconds. We use the following notation to characterize the termination status of the experiments, ``$-$'' means that the method failed to satisfy the stopping criterion within the prescribed maximum number of outer iteration steps; in the reported cases, the residual either stagnated or increased rapidly.
	
	We mainly compare the Safeguarded Anderson-accelerated reflection method (SAAR($m$)) with the reflection method, the restarted minimal projected residual algorithm (RMPR($\tilde{m}$)) \cite{Hu1992}, the Global FOM-Sylvester-Like (GFSL($\tilde{m}$)) \cite{SalkuyehToutounian2006}, and the gradient iteration (GI) \cite{DingLiuDing2008}. RMPR($\tilde{m}$) and GFSL($\tilde{m}$) are based on Krylov subspace methods, and $\tilde{m}$ is the dimensions of Krylov subspaces and also denotes their restart parameter.
	
	\noindent\textbf{Example 1.} For this example, we use the matrices
	\[
	A = B = M + 2rN + \frac{100}{(n+1)^2}I,
	\]
	where $M, N \in \mathbb{R}^{n\times n}$ are the tridiagonal matrices given by
	\[
	M = \text{tridiag}(-1, 4, -1) \quad \text{and} \quad N = \text{tridiag}(0.5, 0, -0.5)
	\]
	and the parameter is set to $r=0.01$. We test dimensions $n = 32, 64, 128, 256$.
	
	We test the Anderson window sizes
	\[
	m=0,1,2,3,4,5,
	\]
	in particular, when \(m=0\), the method reduces to the standard reflection method without Anderson acceleration.
	
	The results for different Anderson window sizes and dimensions are reported in Table~\ref{tab:example1}.
	
	\begin{table}[H]
		\centering
		\caption{Number of iterations (Out-Itr, Inner-Itr) and CPU time (seconds) for different window sizes $m$ and dimensions $n$}
		\label{tab:example1}
		\begin{tabular}{c *{3}{S[table-format=4.0] S[table-format=4.0] S[table-format=3.3]}}
			\toprule
			\multirow{2}{*}{$n$} & \multicolumn{3}{c}{$m=0$} & \multicolumn{3}{c}{$m=1$} & \multicolumn{3}{c}{$m=2$} \\
			\cmidrule(lr){2-4} \cmidrule(lr){5-7} \cmidrule(lr){8-10}
			& {Out-Itr} & {Inner-Itr} & {CPU Time} & {Out-Itr} & {Inner-Itr} & {CPU Time} & {Out-Itr} & {Inner-Itr} & {CPU Time} \\
			\midrule
			32  & 68  & 1325 & 0.034 & 18 & 1183 & 0.047 & 12 & 969 & 0.025 \\
			64  & 68  & 2698 & 0.174 & 16 & 2042  & 0.128 & 12 & 1986  & 0.123 \\
			128 & 66  & 5261 &  1.760 & 16 & 4091 & 1.346 & 12 & 4064 & 1.264 \\
			256 & 65 & 10399 &  13.857 & 16& 8158 & 10.769 & 12 & 8160 & 11.245 \\
			\bottomrule
		\end{tabular}
	\end{table}
	
	\begin{table}[H]
		\centering
		\caption*{Table \ref{tab:example1} (continued)}
		\begin{tabular}{c *{3}{S[table-format=4.0] S[table-format=4.0] S[table-format=3.3]}}
			\toprule
			\multirow{2}{*}{$n$} & \multicolumn{3}{c}{$m=3$} & \multicolumn{3}{c}{$m=4$} & \multicolumn{3}{c}{$m=5$} \\
			\cmidrule(lr){2-4} \cmidrule(lr){5-7} \cmidrule(lr){8-10}
			& {Out-Itr} & {Inner-Itr} & {CPU Time} & {Out-Itr} & {Inner-Itr} & {CPU Time} & {Out-Itr} & {Inner-Itr} & {CPU Time} \\
			\midrule
			32  & 13  & 910 & 0.018 & 12 & 840 & 0.022 &11 & 815 & 0.016\\
			64  & 13  & 1809 & 0.112 & 12 & 1693 & 0.107 & 11 & 1642 & 0.114 \\
			128 & 11  & 3008  & 0.946 & 10 & 2807 & 0.904 &11 & 3275 &  1.011\\
			256 & 11 & 5995 & 7.679 &10&  5619  & 6.660 & 11 &  6450 & 7.903\\
			\bottomrule
		\end{tabular}
	\end{table}
	
	Table~\ref{tab:example1} shows that the Anderson-accelerated variants require substantially fewer outer iterations than the reflection method for all tested dimensions. In particular, the number of outer iterations is reduced from 65--68 for the unaccelerated method to 10--18 for the accelerated variants. The reduction in the total number of inner iterations is also evident. Although the CPU time depends on both the number of inner iterations and the cost of the Anderson least-squares step, the accelerated schemes generally provide a favorable computational performance. These results demonstrate the practical effectiveness of the safeguarded Anderson acceleration for the selected test problem.
	
	\noindent\textbf{Influence of the inner tolerance.} We next investigate the influence of the inner tolerance $\varepsilon_\text{in}$ for Example~1 with $A = B$ and $n = 32$. The maximum number of outer iterations was set to 200. Table~\ref{tab:iteration_data} reports the outer iteration counts and CPU times, while Figure~\ref{fig:residual_combined} shows the corresponding relative Sylvester residual histories for $m = 0,1,\dots,5$.
	
	The unaccelerated method is highly sensitive to the inner tolerance. It converges for $\varepsilon_\text{in} = 0.1$, but fails to satisfy the stopping criterion for smaller tolerances. This behavior is consistent with the discussion in Remark~2: when the inner solve becomes more accurate, the outer iteration approaches the exact fixed-point iteration, which is nonconvergent for $A = B$ and $X^{(0)} = 0$.
	
	For the Anderson-accelerated variants, convergence is much more robust. In particular, for $m \ge 2$, all tested inner tolerances lead to convergence within a small number of outer iterations. Nevertheless, the behavior is not completely monotone with respect to $m$ or $\varepsilon_\text{in}$; for example, the case $m = 1$ and $\varepsilon_\text{in} = 10^{-2}$ fails. These observations provide numerical evidence that safeguarded Anderson acceleration improves practical robustness, although they do not constitute a general convergence proof.
	
	\begin{table}[H]
		\centering
		\renewcommand{\arraystretch}{1.2}
		\caption{Outer iterations and CPU times for different inner tolerances $\varepsilon_\text{in}$ and Anderson window sizes $m$ in the case $A = B$ with $n=32$.}
		\label{tab:iteration_data}
		\vspace{0.2cm}
		\begin{tabular}{c c c c c c c c c}
			\toprule
			\multirow{2}{*}{\(m\)} & \multicolumn{2}{c}{$\varepsilon_\text{in}=0.1$} & \multicolumn{2}{c}{$\varepsilon_\text{in}=0.01$} & \multicolumn{2}{c}{$\varepsilon_\text{in}=0.001$} & \multicolumn{2}{c}{$\varepsilon_\text{in}=1\times 10^{-6}$} \\
			\cmidrule(lr){2-3} \cmidrule(lr){4-5} \cmidrule(lr){6-7} \cmidrule(lr){8-9}
			& Out-Itr & CPU Time  & Out-Itr & CPU Time  & Out-Itr & CPU Time & Out-Itr & CPU Time  \\
			\midrule
			0 & 68 & 0.034  & $-$  & $-$  & $-$  & $-$  & $-$  & $-$  \\
			1 & 18  & 0.047 & $-$  & $-$  & 8  & 0.036  & 4  & 0.042 \\
			2 & 12  & 0.025 & 13  & 0.045& 7  & 0.036& 3   & 0.041 \\
			3 & 13  & 0.018& 11  & 0.030 & 6  & 0.025 & 3   & 0.027 \\
			4 & 12  & 0.022& 9  & 0.025 & 5  & 0.021 & 3   & 0.028 \\
			5 & 11  & 0.016& 7  & 0.022& 5  & 0.022& 3   & 0.026 \\
			\bottomrule
		\end{tabular}
	\end{table}
	
	\begin{figure}[H]
		\centering
		\includegraphics[width=0.95\textwidth]{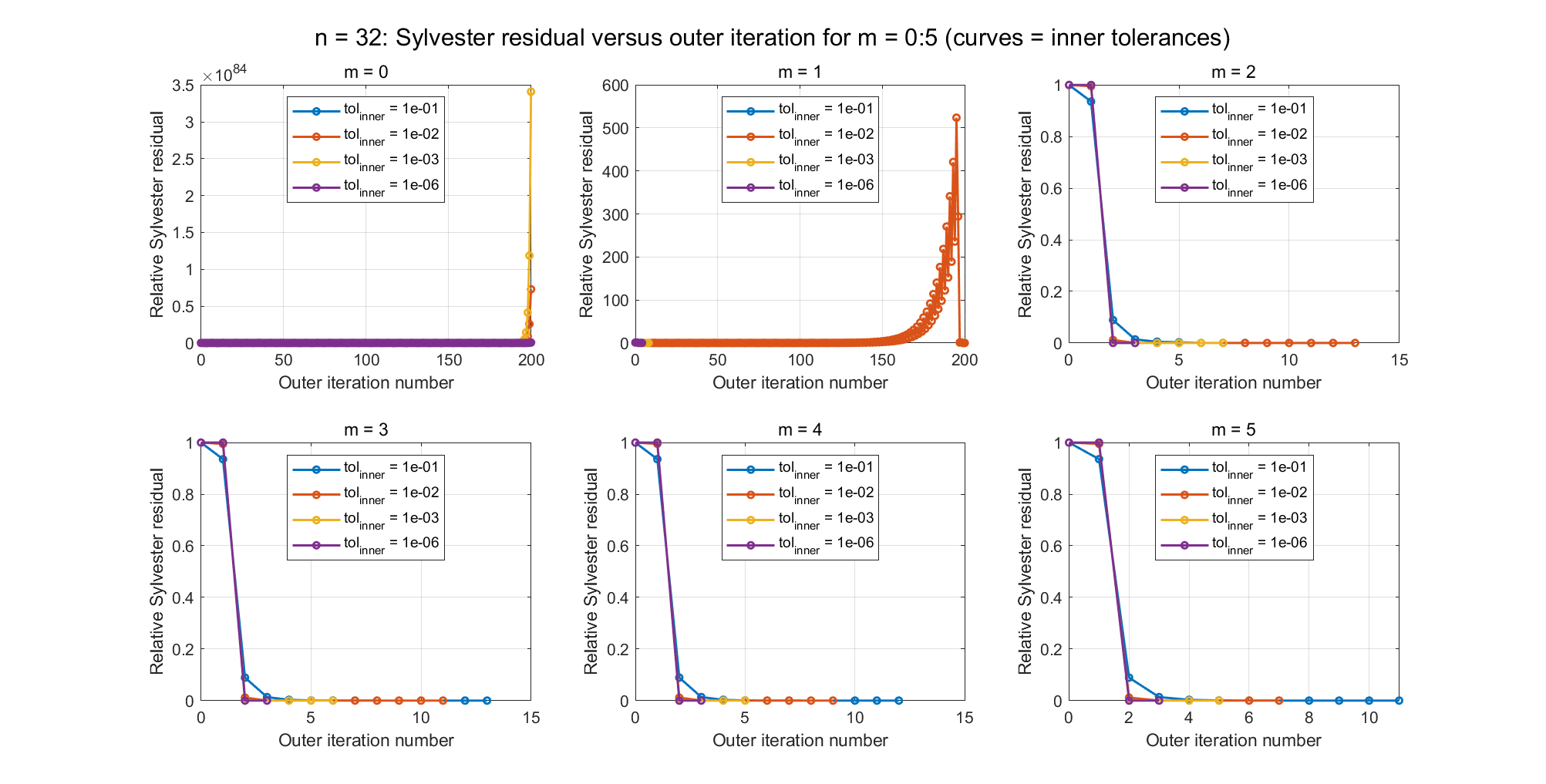}
		\caption{Relative Sylvester residual versus outer iteration number for $n = 32, m = 0, 1,\ldots, 5$, and different inner tolerances.}
		\label{fig:residual_combined}
	\end{figure}
	
	\noindent\textbf{Example 2.} For the second example, we consider
	\[
	A=
	\begin{pmatrix}
		\frac{\sqrt{3}}{2} & -\frac{1}{2} & 0 \\
		\frac{1}{2} & \frac{\sqrt{3}}{2} & 0 \\
		0 & 0 & 1
	\end{pmatrix},
	\]
	and
	\[
	B=\text{tridiag}(0.001, 0, 0.001).
	\]
	We consider the case $n=3.$
	
	We apply the iteration methods to this problem, and the results are given in Table~\ref{tab:example2}.
	
	\begin{table}[H]
		\centering
		\caption{Numerical experiments for Example 2 with $n=3$.}
		\label{tab:example2}
		\begin{tabular}{cccc}
			\toprule
			Method & Out-Itr & CPU Time & Res-Norm \\
			\midrule
			reflection method  & 5  & 0.047 & $6.284\times 10^{-7}$ \\
			RMPR(3) & 1 & 0.156 & $2.184\times 10^{-16}$ \\
			GFSL(3)  & 3 & 0.063 & $2.309\times 10^{-9}$ \\
			GI & 7 & 0.031 & $1.016\times10^{-7}$ \\
			SAAR(3) & 3 & 0.023 & $5.205\times 10^{-8}$ \\
			\bottomrule
		\end{tabular}
	\end{table}
	
	From Table~\ref{tab:example2}, for this small MATLAB test, SAAR(3) requires fewer iterations than the standard reflection method and GI, and also gives the shortest recorded CPU time among the tested methods. RMPR(3) converges in one outer iteration and attains the smallest residual, but its CPU time is higher for this small problem.
	
	\noindent\textbf{Example 3.} In the third example, we consider the coefficient matrices
	\[
	\begin{cases}
		A = \operatorname{diag}(1, 2, \dots, n) + rL^T,\\
		B = 2^{-t}I_n + rL^T + 2^{-t}L,
	\end{cases}
	\]
	with $n=100$, where $L$ is the strictly lower triangular matrix whose entries below the main diagonal are all equal to one \cite{Bai2011}. Here, we set parameters $t = 10$ and $r = 0.0001$.
	
	Table~\ref{tab:example3} presents the performance of different methods for this problem.
	
	\begin{table}[H]
		\centering
		\caption{Numerical experiments for Example 3 with $n=100$.}
		\label{tab:example3}
		\begin{tabular}{cccc}
			\toprule
			Method  & Out-Itr & CPU Time & Res-Norm \\
			\midrule
			reflection method  & 6 & 0.552 & $9.961\times 10^{-7}$\\
			RMPR(5) & $-$ & $-$ & $-$\\
			GFSL(5)  & 75215 & 184.550 & $9.999\times 10^{-7}$ \\
			GI & 75118 & 185.819 &  $9.999\times 10^{-7}$\\
			SAAR(5)   & 3 & 0.269 & $2.228\times 10^{-7}$ \\
			\bottomrule
		\end{tabular}
	\end{table}
	From Table~\ref{tab:example3}, the reflection method and SAAR(5) converge in 6 and 3 outer iterations, respectively. SAAR(5) requires about half the CPU time of the reflection method and attains the smallest residual among the tested methods. GFSL(5) and GI require about $7.5\times 10^4$ iterations under the chosen parameters, indicating very slow convergence for this problem. RMPR(5) does not satisfy the stopping criterion within the prescribed iteration limit; this may be related to the small projection dimension and the absence of an evident low-rank structure in the solution. Overall, for this test problem and this implementation, SAAR(5) is the most efficient among the compared methods.
	
	\section{Conclusions}
	
	In this paper, we proposed a reflection-based two-level iterative method for the continuous Sylvester equation
	\[
	AX+XB=C.
	\]
	The equation was first reformulated as a fixed-point iteration. At each outer iteration, the resulting linear matrix equation was decomposed into independent column-wise linear systems, and these systems were solved approximately by the reflection method. This gives a practical nested iteration in which the fixed-point iteration acts as the outer iteration and the reflection method serves as the inner solver.
	
	For the inner reflection iteration, we derived its compact matrix form and proved, by using a spectral-radius argument, that it converges to the exact solution of the corresponding linear system when \(A\) is nonsingular and \(n\ge 2\). By stacking the column-wise iterations, this convergence result was extended to the inner matrix equation. Since the inner problems are solved only approximately in practical computations, the whole procedure was further interpreted as an inexact fixed-point iteration. Under a residual-based inner stopping condition, we established a sufficient condition for linear convergence of the unaccelerated inexact outer iteration in the Frobenius norm.
	
	To improve the practical convergence behavior, we incorporated Anderson acceleration into the outer iteration and used a residual-based safeguard. The safeguard rejects Anderson candidates whose Sylvester residual is larger than that of the standard inexact reflection candidate. The numerical experiments show that the safeguarded Anderson-accelerated reflection method can significantly reduce the number of outer iterations and often decrease the total CPU time in the reported MATLAB tests. In particular, the tolerance experiments demonstrate that the unaccelerated method may be sensitive to the inner tolerance in borderline cases, while the Anderson-accelerated variants are generally more robust for the tested problems.
	
	It should be emphasized that the convergence theorem in this paper applies only to the unaccelerated inexact fixed-point iteration under the condition \(q<1\). The convergence of the safeguarded Anderson-accelerated scheme is supported by numerical evidence, but a complete theoretical convergence analysis remains open. Moreover, the experiments and the discussion in Remark~2 show that the practical behavior of the method depends on the interaction among the inner stopping rule, the inexact fixed-point map, and the Anderson extrapolation. In particular, convergence observed outside the sufficient condition of Theorem~\ref{thm:main-convergence} should not be interpreted as a consequence of that theorem.
	
	Future work includes developing convergence conditions for the safeguarded Anderson-accelerated scheme, improving the treatment of rank-deficient or ill-conditioned Anderson least-squares problems, and extending the proposed framework to large-scale sparse problems, low-rank formulations, generalized Sylvester equations, and tensor matrix equations.

\end{document}